\documentclass[12pt]{amsart}
\usepackage{enumerate}
\usepackage{bbm}
\usepackage{graphicx}
\usepackage{mathrsfs}
\usepackage{amsmath, amssymb, amsthm, latexsym}
\usepackage{amssymb,amscd,amsmath}
\usepackage{latexsym}
\usepackage{MnSymbol}
\usepackage{enumitem}
\usepackage[center]{caption}
\usepackage{tikz}
\usepackage[all]{xy}
\usepackage{amsxtra,mathtools,stmaryrd}
\newcounter{braid}
\newcounter{strands}

\DeclareMathAlphabet{\bsf}{OT1}{cmss}{bx}{n}

\pgfkeyssetvalue{/tikz/braid height}{1cm}
\pgfkeyssetvalue{/tikz/braid width}{1cm}
\pgfkeyssetvalue{/tikz/braid start}{(0,0)}
\pgfkeyssetvalue{/tikz/braid colour}{black}
\pgfkeys{/tikz/strands/.code={\setcounter{strands}{#1}}}

\makeatletter
\def\cross{%
  \@ifnextchar^{\message{Got sup}\cross@sup}{\cross@sub}}

\def\cross@sup^#1_#2{\render@cross{#2}{#1}}

\def\cross@sub_#1{\@ifnextchar^{\cross@@sub{#1}}{\render@cross{#1}{1}}}

\def\cross@@sub#1^#2{\render@cross{#1}{#2}}

\def\render@cross#1#2{
  \def\strand{#1}
  \def\crossing{#2}
  \pgfmathsetmacro{\cross@y}{-\value{braid}*\braid@h}
  \pgfmathtruncatemacro{\nextstrand}{#1+1}
  \foreach \thread in {1,...,\value{strands}}
  {
    \pgfmathsetmacro{\strand@x}{\thread * \braid@w}
    \ifnum\thread=\strand
    \pgfmathsetmacro{\over@x}{\strand * \braid@w + .5*(1 - \crossing) * \braid@w}
    \pgfmathsetmacro{\under@x}{\strand * \braid@w + .5*(1 + \crossing) * \braid@w}
    \draw[braid] \pgfkeysvalueof{/tikz/braid start} +(\under@x pt,\cross@y pt) to[out=-90,in=90] +(\over@x pt,\cross@y pt -\braid@h);
    \draw[braid] \pgfkeysvalueof{/tikz/braid start} +(\over@x pt,\cross@y pt) to[out=-90,in=90] +(\under@x pt,\cross@y pt -\braid@h);
    \else
    \ifnum\thread=\nextstrand
    \else
     \draw[braid] \pgfkeysvalueof{/tikz/braid start} ++(\strand@x pt,\cross@y pt) -- ++(0,-\braid@h);
    \fi
   \fi
  }
  \stepcounter{braid}
}

\tikzset{braid/.style={double=\pgfkeysvalueof{/tikz/braid colour},double distance=1pt,line width=2pt,white}}

\newcommand{\braid}[2][]{%
  \begingroup
  \pgfkeys{/tikz/strands=2}
  \tikzset{#1}
  \pgfkeysgetvalue{/tikz/braid width}{\braid@w}
  \pgfkeysgetvalue{/tikz/braid height}{\braid@h}
  \setcounter{braid}{0}
  \let\sigma=\cross
  #2
  \endgroup
}
\makeatother

\input xypic
\newtheorem{theorem}{Theorem}%[section]
\newtheorem{proposition}[theorem]{Proposition}

\def\Z{\mathbb{Z}}

\def\C{\mathbb{C}}

\def\R{\mathbb{R}}
\def\C{\mathbb{C}}

\def\N{\mathbb{N}}

\def\T{\mathbb{T}}

\def\qed{\hfill$\square$\medskip}

\def\Zpk{\mathbb{Z}/p^{k}}
\def\Zpk1{\mathbb{Z}/p^{k-1}}

\newcommand{\rref}[1]{(\ref{#1})}

\newcommand{\beg}[2]{\begin{equation}\label{#1}#2\end{equation}}
\def\r{\rightarrow}

\def\sl2{\widetilde{SL_{2}(\Z)}}

\author{Po Hu, Igor Kriz, Petr Somberg, and Foling Zou}
\title[$BP\R$]{The $\Z/p$-equivariant spectrum $BP\R$ for an odd prime $p$}
\thanks{Hu acknowledges support by NSF grant 2301520. Kriz acknowledges the support of Simons Foundation award 958219. Somberg acknowledges support
from the grant GACR 19-28628X}

\begin{document}

\maketitle

\begin{center} \small {\sc With contributions by Guchuan Li}

\end{center}

\vspace{10mm}

\begin{abstract}
In the present paper, we construct a $\mathbb{Z}/p$-equivariant analog of the $\mathbb{Z}/2$-equivariant
spectrum $BP\mathbb{R}$ previously constructed by Hu and Kriz. We prove that this spectrum has some of
the properties conjectured by Hill, Hopkins, and Ravenel. Our main construction method is an $\mathbb{Z}/p$-equivariant 
analog of the Brown-Peterson tower of $BP$, based on a previous description of the $\Z/p$-equivariant Steenrod
algebra with constant coefficients by the authors. We also describe several variants of our construction and
comparisons with other known equivariant spectra.
\end{abstract}

\vspace{5mm}

\section{Introduction}

Hill, Hopkins, and Ravenel \cite{hhr} conjectured that there exists a $\Z/p$-equivariant structure on
$BP^{\wedge (p-1)}$ for $p>2$ such that the geometric fixed point spectrum is $H\Z/p$. This equivariant
spectrum is expected to have a number of properties. The purpose
of this paper is to give a construction of such a spectrum and prove some of the
properties. Our construction is based on a mild extension of the computation of the $\Z/p$-equivariant Steenrod algebra
with respect to the constant Mackey functor $\underline{\Z/p}$ \cite{st1,sankarw}. Following Sankar-Wilson 
\cite{sankarw}, we denote by $T$ the second desuspension of the $\Z/p$-equivariant degree $p$
map
$$S^\beta\r S^2$$
where $\beta$ is the basic irreducible representation of $\Z/p$. Denoting by $\widetilde{\mathcal{L}}_{p-1}$
the $\Z/p$-Mackey functor whose non-equivariant part is the reduced regular representation 
$\mathcal{L}_{p-1}$ of
$\Z/p$ and the $\Z/p$-fixed part is $0$, there is a map of $\Z/p$-equivariant spectra
\beg{esteenrodxin}{
H\underline{\Z}\r \Sigma^{2p^{n-1}+(p^n-p^{n-1}-1)\beta} H\underline{\Z}\wedge T
}
(which comes from the fact that after smashing with $H\underline{\Z}$, the right-hand side of \rref{esteenrodxin}
becomes a summand of the left-hand side as an $H\underline{\Z}$-module.
Now we also have a ``connecting map"
$$H\underline{\Z}\wedge T\rightarrow \Sigma^{\beta-1}
H\widetilde{\mathcal{L}}_{p-1}$$
(see the cofibration sequence \rref{eresol1} below). Composing with \rref{esteenrodxin}, we get a map
\beg{ekinv}{Q_n^\prime:H\underline{\Z}\r \Sigma^{2p^{n-1}-1+
(p^n-p^{n-1})\beta}H\widetilde{\mathcal{L}}_{p-1}}
(see formula \rref{eresol2} below).
This map is a $\Z/p$-equivariant analog of the integral $Q_n$-elements which form the first k-invariant
of the Brown-Peterson construction of $BP$ \cite{bp}. 

The approach of our construction is to construct a $\Z/p$-equivariant spectrum $BP\R$ by mimicking, in a minimal
way, the Brown-Peterson construction \cite{bp} in the category of $\Z/p$-equivariant spectra. It is important to note that the appearance of the non-constant Mackey 
functor occurring on the right-hand side of \rref{ekinv} is a reflection of the fact that in $(BP\R_{\{e\}})_*$,
$v_n$ lies in a copy of the representation $\mathcal{L}_{p-1}$. We denote the generator of this representation
by $r_n$. In fact, in the equivariant analogue of the $BP$ tower, a key feature is the interplay between
the Mackey functors $\underline{\Z}$, $\widetilde{\mathcal{L}}_{p-1}$ and the
Mackey functor $\underline{\mathcal{L}}_p$, which is the principal projective 
Mackey functor on a fixed element (i.e. has the integral regular representation as the non-equivariant part, and
$\Z$ as the fixed points).

The non-equivariant
spectrum underlying our ``minimal" $BP\R$ following the Brown-Peterson 
construction is somewhat smaller than the $\bigwedge_{p-1}BP$ conjectured by
Hill, Hopkins, and Ravenel \cite{hhr}. In fact, at each $v_n$, we are forced to put into the non-equivariant
coefficients the representations
\beg{enoneqs}{(\Z\oplus \mathcal{L}_{p-1}r_n\oplus r_n^2\mathcal{L}_{p}[r_n])\otimes \Z[Nr_n]}
where $\mathcal{L}_p$ is the integral regular $\Z/p$-representation and $N$ denotes the multiplicative norm
(see formula \rref{ebpre} below). 

The non-equivariant spectrum conjectured by Hill, Hopkins, and Ravenel \cite{hhr} (which we denote
by $BP\R^{HHR}$) should have, instead of \rref{enoneqs},
\beg{econcc}{Sym(\mathcal{L}_{p-1}\cdot r_n).}
This, in fact, coincides with \rref{enoneqs} for $p=3$. For $p\geq 5$, \rref{econcc} is bigger, but one can
construct a candidate for the spectrum $BP\R^{HHR}$ by forming a wedge of $BP\R$ with a wedge
of even suspensions of copies of the multiplicative norm of $BP$ from $\{e\}$ to $\Z/p$.
(See Section \ref{sjw} below.)

We should remark that we can currently only make the equivariant analogue of the Brown-Peterson 
construction in the category of Borel-complete spectra. In this category, the obstructions to continuing the
construction are either non-equivariant (and vanish for the same reason as in \cite{bp}, i.e. evenness), or
are non-torsion with respect to multiplication by the class represented by the inclusion
$$S^0\r S^\beta,$$
which can therefore be treated on the level of geometric fixed points. On geometric fixed points, however,
the equivariant Brown-Peterson tower splits into a wedge equivalences on summands, which is why
those obstructions also vanish.

The fact that we work in the Borel-complete category, however, should not be a substantial restriction,
since the spectrum $BP\R$ should be Borel-complete anyway (similarly as for $p=2$). In fact, we
have a $\Z/p$-equivariant analog of the Borel cohomology spectral sequence \cite{hk}, which 
we present in Section \ref{sborel}. This allows us to calculate the 
$RO(\Z/p)$-graded coefficients $BP\R_\star$ completely (Theorem \ref{t1} below).
Not having a ``genuine" version of equivariant Brown-Peterson
construction amounts to not having, at the moment, an analogue of the slice spectral sequence.

In Section \ref{sjw}, we also discuss analogues of the $BP\R$ construction on Johnson-Wilson-type spectra.
Using the obstruction theory of Robinson \cite{robinson} and Baker \cite{baker}, 
we construct a $\Z/p$-action related to $BP\R$ on a completion of the smash product of $(p-1)$ copies of $E(n)$
(or its variants, such as $E_n$). We conjecture that there is an ``orientation" map from $BP\R^{HHR}$
to these spectra (however, we do not yet have a ring structure on the source). 

It is also known that $\Z/p$ acts on $E_n$ when $(p-1)\mid n$ via the subgroup of
the Morava stabilizer group. The case
$n=p-1$ has been especially studied
(see \cite{nave,symonds}). It is reasonable to conjecture that $BP\R^{HHR}$ maps into this equivariant
$E_n$ in a way which is identity on the $\mathcal{L}_{p-1}$ containing $v_n$ in
the non-equivariant coefficients, for which we give evidence on the level of formal group laws.

The present paper incorporates numerous observations, suggestions, and references communicated to us by Guchuan Li.

\vspace{5mm}

\section{Preliminary computations}

We will use the 
integral trivial, regular resp.
reduced regular representation $\mathcal{L}_1$, $\mathcal{L}_p$, $\mathcal{L}_{p-1}$
of $\Z/p$. Note that all are isomorphic to their duals integrally.
In case of $\mathcal{L}_{p-1}$, we have an isomorphism
$$\Z^p/(1,1,\dots,1)\r \{(a_1,\dots,a_p)\in \Z^p\mid \sum a_i=0\}$$
by sending
$$(1,0,\dots,0)\mapsto (1,-1,0,\dots,0).$$
We will work with integral Mackey functors here. We denote by $\underline{\Z}=\underline{\mathcal{L}}_1$ the constant 
Mackey functor, and by $\widetilde{\mathcal{L}}_{p-1}$ the Mackey functor equal to the integral reduced regular representation 
on the free orbit and $0$ on the fixed orbit. We also have the co-constant Mackey functor $\overline{\mathcal{L}}_1$ where again both the non-equivariant and fixed parts are $\Z$ and
the corestriction is $1$ while the restriction is $p$. We have short exact sequences
\beg{ess1}{0\r \underline{\mathcal{L}}_1\r \underline{\mathcal{L}}_p\r \widetilde{\mathcal{L}}_{p-1}\r 0}
\beg{ess2}{0\r \widetilde{\mathcal{L}}_{p-1}\r \underline{\mathcal{L}}_p\r\overline{\mathcal{L}}_1\r 0.}
We also note that for a non-trivial irreducible representation $\beta$ of $\Z/p$ ($p>2$), we have
\beg{ess3}{H\overline{\mathcal{L}}_1=\Sigma^{2-\beta}H\underline{\mathcal{L}}_1.}
We also have 
\beg{esh1}{H\widetilde{\mathcal{L}}_{p-1}\wedge_{H\underline{\Z}}H\widetilde{\mathcal{L}}_{p-1}=
H\overline{\mathcal{L}}_1\vee\bigvee_{p-2}H\underline{\mathcal{L}}_p.
}
We will work $p$-locally, so whenever we say ``$\Z$," we mean $\Z_{(p)}$. Then {\em all} equivariant spectra
are $(\beta-\beta^{\prime})$-periodic for non-trivial irreducible $\Z/p$-representations $\beta,\beta^\prime$. This
is due to the fact that there exists a $\Z/p$-equivariant map of spaces
$$S^\beta\r S^{\beta^\prime}$$
of some degree $k\in \Z/p^\times$, which is a $p$-local equivalence (cf. \cite{sankarw}). 
Thus, instead of $RO(\Z/p)$-grading,
we can consider $R$-grading where $R=\Z\{1,\beta\}$ for a chosen non-trivial irreducible $\Z/p$-representation 
$\beta$. This is also true for ordinary $\Z/p$-equivariant homology with coefficients over a $\underline{\Z}$-Mackey 
module for a different reason (see \cite{st1}).

\begin{proposition}\label{p1}
1. The $R$-graded coefficients of $H\underline{\mathcal{L}}_1$ are $\Z$ in degrees $2k-k\beta$, $k\geq 0$ 
and $\Z/p$ in degrees $2k-\ell\beta$ with $0\leq k<\ell$ (this is called the {\em good wedge})
and $\Z$ in degrees $2k-k\beta$, $k<0$ and $\Z/p$ in degrees $-1-2k+\ell\beta$ where $2<2k+1<2\ell$
(this is called the {\em derived wedge}).

2. The $R$-graded coefficients of $H\widetilde{\mathcal{L}}_{p-1}$ are $\Z/p$ in degrees $2k+1-\ell\beta$ 
where $0<2k+1<2\ell$ (this is the {\em good wedge}) and $\Z/p$ in degrees $-2k+\ell\beta$ where
$0<k\leq\ell$ (this is the {\em derived wedge}).

(In the remaining $R$-dergrees, the coefficients are $0$.)

\end{proposition}
\qed

\vspace{3mm}
The pattern of the $R$-graded coefficients of $H\underline{\mathcal{L}}_1=H\underline{\Z}$ is shown in Figure 1 below (where a solid dot 
means a copy of $\Z/p$ and an empty square means a copy of $\Z$.

\begin{figure}
\includegraphics{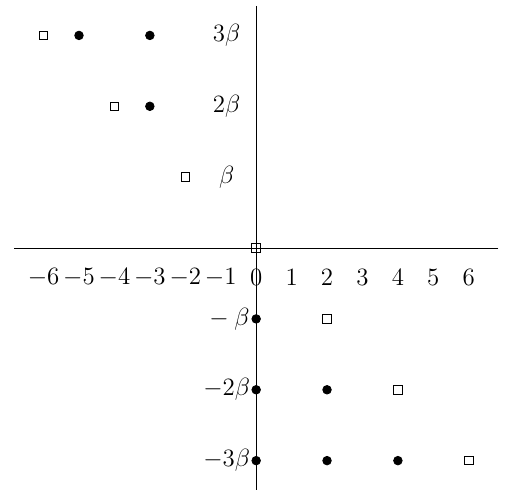}
\caption{The coefficients $H\underline{\Z}_\star$}
\end{figure}

In our tower illustration, we shall use the shortcut for this pattern shown in Figure 2.

\begin{figure}
\includegraphics{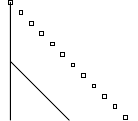}
\caption{Shorthand notation for $H\underline{\Z}_\star$}
\end{figure}

Figure 3 shows the pattern of the $R$-graded coefficients of $\Sigma^{\beta-1}H\widetilde{\mathcal{L}}_{p-1}$. 
The shift is so that the corner of the ``good wedge" is at the $(0,0)$-point.

\begin{figure}
\includegraphics{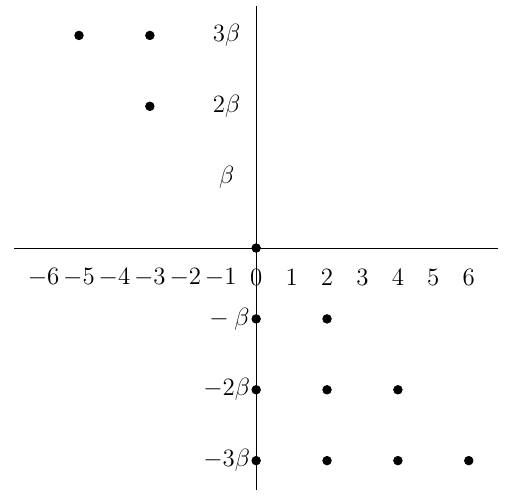}
\caption{The coefficients $(\Sigma^{\beta-1}H\widetilde{\mathcal{L}}_{p-1})_\star$}
\end{figure}

Figure 4 shows the shortcut we use for this pattern.

\begin{figure}
\includegraphics{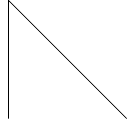}
\caption{Shorthand notation for $(\Sigma^{\beta-1}H\widetilde{\mathcal{L}}_{p-1})_\star$}
\end{figure}

Figure 5 shows the $R$-graded coefficients of $H\underline{\Z/p}$.

\begin{figure}
\includegraphics{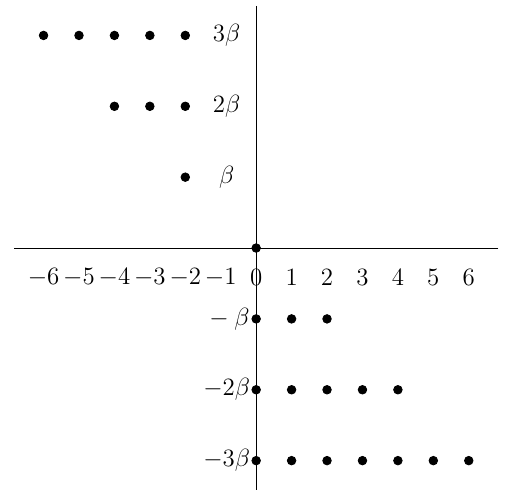}
\caption{The coefficients $H\underline{\Z/p}_\star$}
\end{figure}

Our shortcut for this pattern is shown in Figure 6.

\begin{figure}
\includegraphics{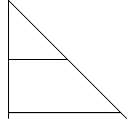}
\caption{Shorthand notation for $H\underline{\Z/p}_\star$}
\end{figure}

Figure 7 shows the $R$-graded coefficients of 
$\Sigma^{\beta-1}H\widetilde{L}_{p-1}=\Sigma^{\beta-1}H(\widetilde{L}_{p-1}/p)$.

\begin{figure}
\includegraphics{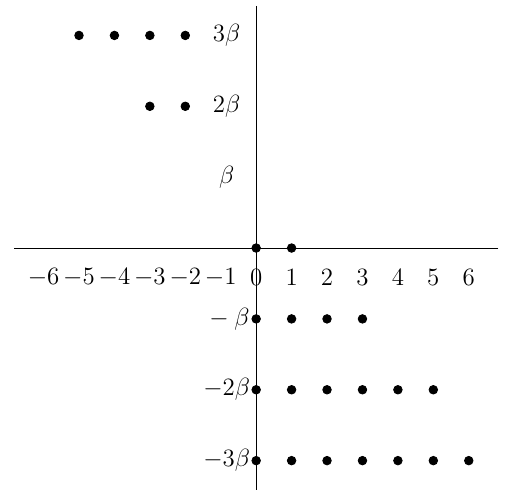}
\caption{The coefficients $(\Sigma^{\beta-1}H(\widetilde{\mathcal{L}}_{p-1}/p))_\star$}
\end{figure}

Our shorthand for this pattern is shown in Figure 8.

\begin{figure}
\includegraphics{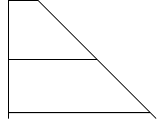}
\caption{Shorthand notation for $(\Sigma^{\beta-1}H(\widetilde{\mathcal{L}}_{p-1}/p))_\star$}
\end{figure}

\vspace{3mm}

\section{The Borel cohomology spectral sequence}\label{sborel}

The Borel cohomology $F(E\Z/p_+,X)$ for a $\Z/p$-equivariant spectrum $X$ will be denoted
by $X^c$. Coefficients of $H\underline{\mathcal{L}}_1^c$, $H\widetilde{\mathcal{L}}_{p-1}^c$ are obtained by 
inverting $\sigma^{-2}$ in the good wedge in Proposition \ref{p1} (which acts by isomorphism
wherever dimensionally possible). 

Homotopy classes in a summand of the non-equivariant homotopy groups of a $\Z/p$-equivariant
spectrum which is isomorphic to $\mathcal{L}_p$ is called {\em negligible}. Such classes cannot receive or support
differentials in the Borel cohomology spectral sequence because of the theory of Mackey functors.

We will completely calculate the Borel cohomology spectral sequence of $BP\R$ at an odd prime $p$
modulo negligible $BP$-summands in $BP\R_{\{e\}}$. Modulo negligible $BP$-summands, $(BP\R_{\{e\}})_*$
can be written as
\beg{ebpre}{
B=\bigotimes_{n>0} (\mathcal{L}_1\oplus \mathcal{L}_{p-1}\cdot r_n\oplus r_n^2\mathcal{L}_p[r_n])[\Phi(v_n)]
}
where $\Phi(v_n)$ has the same dimension as $v_n^p$ and
$r_n$ is a generator of a $\mathcal{L}_{p-1}$-subrepresentation which contains $v_n$. 

\vspace{3mm}

Recall that in \cite{hk}, Hu and Kriz started with the $E^1$-term of the Borel cohomology spectral sequence and included $d^1$
in the differential pattern. This has the advantage of allowing to treat $v_0$ on the same level as the other $v_n$'s, which makes the
pattern more natural.

To describe the appropriate analog for $p$ odd, we preview the slightly more complicated assortment of elements and
differentials \rref{ediff1}, \rref{evn}, \rref{ephivn} we will encounter. From the point of view of this pattern, the most natural
way to start is to put 
\beg{ephivp}{p=\Phi(v_0),}
and include the second, but not the first, differential of \rref{ediff1} for $n=0$. Thus, the elements $v_0$, $\sigma^{\frac{2}{p}-2}$ of
fractional degrees
$$|v_0|=\frac{2}{p}-1-\frac{\beta}{p},\; |\sigma^{\frac{2}{p}-2}|=\beta(\frac{1}{p}-1)-\frac{2}{p}+2$$
are not present in our spectral sequence, but the exterior generator 
$$v_0\sigma^{\frac{2}{p}-2}$$
of degree
$$|v_0\sigma^{\frac{2}{p}-2}|=1-\beta$$
is included, making the $E^1$-term
$$B[\sigma^2,\sigma^{-2}][b]\otimes\Lambda(v_0\sigma^{\frac{2}{p}-2}).$$
The $d^1$ differential has 
$$d^1(v_0\sigma^{\frac{2}{p}-2})=pb,$$
but also takes group cohomology on \rref{ebpre}. This gives $E^2$-term
\beg{ebpre1}{E^2 = B'[\sigma^2,\sigma^{-2}][b]/pb}
where if we put
\beg{evi}{v_I=\prod v_{n}^{i_n}, \; I=(i_1,i_2,\dots),\; i_n\in \{0,1\},}
(note that $v_n^2$ is negligible), and
\beg{evi1}{|I|=\sum_n i_n,
}
then 
$$B' = \bigl( \mathbb{Z}\{v_I: I \text{ even}\} \oplus \mathbb{Z}/p\{v_I: I \text{ odd}\}\bigr)[\Phi(v_n)].$$

The geometric fixed points of $BP\R$ are $H\Z/p$, so we need a resolution with respect to the dual
of the polynomial algebra $\Z[\sigma^{-2}]$, which is a divided power polynomial algebra. This implies two sets
of differentials
\beg{ediff1}{\sigma^{-2p^{n-1}}\mapsto v_nb^?,\; v_n\sigma^{-2(p-1)p^{n-1}}\mapsto \Phi(v_n)b^?.}
To figure out the exact power of $b$ in \rref{ediff1}, we need to need to figure out the exact equivariant degree
of $v_n$. The $\Z$-graded component must be $2p^{n-1}-1$ in order for the differential \rref{ediff1} to decrease
degree by $1$. Given the fact that $v_n$ sits in a $\widetilde{\mathcal{L}}_{p-1}$, and given Proposition \ref{p1},
its equivariant degree should be of the form $k+\ell\beta$ where $k+2\ell=|v_n|-1=2p^n-3$. Thus, we conclude
that
\beg{evn}{
|v_n|=2p^{n-1}-1+(p^n-p^{n-1}-1)\beta.
}
Similarly, given the fact that $\Phi(v_n)$ sits in a copy of $\underline{\mathcal{L}}_1$, we have
\beg{ephivn}{
|\Phi(v_n)|=2p^n-2+(p^n-1)(p-1)\beta.
}
From this, we conclude that the differential pattern is
\beg{ediffeven}{\sigma^{-2p^{n-1}}\mapsto v_nb^{p^n-1},\; v_n\sigma^{-2(p-1)p^{n-1}}\mapsto \Phi(v_n)b^{
(p^n-1)(p-1)+1}.}
(The differentials on $\sigma^{-2kp^{n-1}}$ for $2\leq k\leq p-1$ are determined by the Leibniz rule.)
This is called the {\em even differential pattern}. 

\vspace{3mm}
\noindent

\vspace{3mm}

The reason this does not tell the whole story is that we have \rref{esh1}, so 
$v_I$ sits in a $\sigma^{2\ell}$-shifted copy of $\widetilde{\mathcal{L}}_{p-1}$ resp. $\underline{\mathcal{L}}_1$
depending on whether $|I|$ is even resp. odd. In fact, working out the shifts precisely, one gets that
\beg{evii}{
|v_I|=\sum_{i_n=1}(2p^{n-1}-1)+
\left(
\sum_{i_n=1}(p^n-p^{n-1})-\lceil\frac{|I|}{2}\rceil
\right)\beta.
}
When multiplying \rref{ediffeven} by $v_I$ with $i_n=0$ and $|I|$ odd, the $\sigma^{-2}$-power will be
in a copy of $\widetilde{\mathcal{L}}_{p-1}$, (so non-equivariantly, its dimension goes down by $1$), 
while the $v_n$ will be in a copy of $\underline{\mathcal{L}}_1$, so non-equivariantly, its dimension goes up by $1$.
Since the non-equivariant dimension of $b$ is $-2$, we then obtain the {\em odd differential pattern}
\beg{ediffodd}{\sigma^{-2p^{n-1}}\mapsto v_nb^{p^n},\; v_n\sigma^{-2(p-1)p^{n-1}}\mapsto \Phi(v_n)b^{
(p^n-1)(p-1)}.}
Of course, in the spectral sequence, the differential of every monomial has multiple summands, and we must
select the summand which generates lowest $b$-torsion. Since, luckily,
$$p^n<(p^n-1)(p-1)$$
and
$$(p^n-1)(p-1)+1<p^{n+1}-1,$$
the lowest $b$-torsion is always generated by $v_n$ or $\Phi(v_n)$ with the lowest $n$. For this reason, the
Tate spectral sequence (obtained by inverting $b$) converges to $\Z/p[b,b^{-1}]$ and similarly as in the
case of $p=2$, writing monomials
$$v_I\Phi(v_J)$$
where $I$ is as above and $J=(j_1,j_2,\dots)$, $j_n\in \N_0$,
$$\Phi(v_J)=\prod_n \Phi(v_n)^{j_n},$$
letting $r=r(I)$ resp. $s=s(J)$ be the lowest $n$ for which $i_n\neq 0$ resp. $j_n\neq 0$,
(set to $\infty$ when not applicable),
we obtain the following

\begin{theorem}\label{t1}
Modulo negligible $BP_*[\sigma^2,\sigma^{-2}]$-summands, the $R$-graded coefficients of $BP\R$ are
a sum of 
$$\Z_{(p)}[b],$$
``even-pattern summands" where $|I|$ is odd and $r\leq s$:
$$\bigoplus_{\ell\nequiv-1\mod p}\Z/p\cdot v_I\Phi(v_J)\cdot\sigma^{2p^{r-1}\ell}[b]/(b^{p^r-1})$$
or $|I|$ is even, $I\neq 0$ and $r>s$:
$$\bigoplus_{\ell\in \Z} \Z_{(p)}\cdot v_I\Phi(v_J)\cdot \sigma^{2p^{s}\ell}[b]/(b^{(p^s-1)(p-1)+1},pb)$$
and ``odd-pattern summands" where $|I|$ is even and $I\neq 0$ and $r\leq s$:
$$\bigoplus_{\ell\nequiv-1 \mod p}\Z_{(p)}\cdot v_I\Phi(v_J)\cdot
\Sigma^{2p^{r-1}\ell}[b]/(b^{p^r},pb)$$
or $|I|$ is odd and $r>s$ :
$$\bigoplus_{\ell\in\Z} \Z/p\cdot v_I\Phi(v_J)\cdot \sigma^{2p^s\ell}[b]/(b^{(p^s-1)(p-1)}).$$

\end{theorem}
\qed

\vspace{3mm}

\section{Construction}\label{sconst}

In this section, we will construct $BP\R$ at an odd prime $p$ by realizing spectrally a $\Z/p$-equivariant
analog of the Brown-Peterson resolution. Throughout this section, we will work in Borel cohomology, and
omit this from the notation. With this convention, the integral homology of $BP\R$, modulo negligible
copies of
$$H\Z_*BP[\sigma^2,\sigma^{-2}],$$
is
$$H\underline{\Z}[\underline{\theta}_1,\underline{\theta}_2,\dots]\wedge_{H\underline{\Z}}
\Lambda_{H\widetilde{\mathcal{L}}_{p-1}}[\widehat{\xi}_1,\widehat{\xi}_2,\dots].$$
(The second term denotes the exterior algebra.)
The k-invariants are defined as cohomological operations (in Borel cohomology), and our job is to prove that they
can be realized spectrally (i.e. on the realization of the partial resolution at each step).

While introducing some negligible terms is necessary, the corresponding non-equivariant homology will remain
even-degree and non-torsion, and there is therefore no obstruction to spectral realization of the negligible
k-invariants, by Brown-Peterson's ``even-odd argument" \cite{bp}. On the other hand, the non-negligible 
k-invariants are non-torsion either rationally or non-torsion with respect to $b$ and hence the obstruction
to spectral realization is $0$ because it vanishes rationally, resp. on geometric (i.e., in our case, Tate) fixed points.

\vspace{3mm}

The Borel cohomology version of the $\Z/p$-equivariant Brown-Peterson resolution is essentially described
as a product over $H\underline{\Z}$ of the following sequences:

Starting with $H\underline{\Z}$, we have
\beg{eresol*}{
\begin{array}{l}
H\underline{\Z}\wedge H\underline{\Z}=\\
\bigwedge_{n\geq 1}(H\underline{\Z}\vee (H\underline{\Z}\wedge T\{
\widehat{\xi}_n\underline{\xi}_n^k\mid k=0,\dots,p-2\})\vee H\underline{\Z/p}
\{\widehat{\xi}_n\underline{\xi}_n^{p-1}\})
[\underline{\theta}_n].
\end{array}
}
The $H\underline{\Z/p}$ arises, in the notation of \cite{st1}, as two copies of 
$H\underline{\Z}$ stuck together by $p$-multiplication, in other words from the fact that
$$\beta(\underline{\mu}_n)=\underline{\theta}_n.$$
Now one has a cofibration sequence
\beg{eresol1}{
\Sigma^{\beta-2}H\widetilde{\mathcal{L}}_{p-1}\r H\underline{\Z}\wedge T\rightarrow \Sigma^{\beta-1}
H\widetilde{\mathcal{L}}_{p-1}
}
(split on $R$-graded coefficients). Then, there is an operation 
\beg{eresol2}{
Q_n^\prime:H\underline{\Z}\r \Sigma^\beta H\widetilde{\mathcal{L}}_{p-1}\{v_n\}
}
which applies \rref{eresol1} to the $\underline{\xi}_n$-wedge summand of \rref{eresol*}. (Here
the notation $v_n$ refers to the $R$-graded indexing of the previous section.)

\vspace{3mm}

\noindent
{\bf Comment:} 
The operations $Q_n^\prime$ form the first stage of our equivariant version of the $BP\R$ tower. 
One may wonder how these elements are identified. To this end, it is beneficial to consider what
happens on the level of geometrical fixed points of homology. We have 
\beg{ephicon1}{\begin{array}{l}\Phi^{\Z/p}(H\underline{\Z}\wedge BP\R)_\star=\\[1ex]
\Phi^{\Z/p}(H\underline{\Z})\wedge \Phi^{\Z/p}(BP\R)_*[b,b^{-1}]=\\[1ex]
A_*[\sigma^{-2}][b,b^{-1}].
\end{array}}
On the other hand, we have
\beg{ephicon2}{\begin{array}{l}\Phi^{\Z/p}(H\underline{\Z}\wedge H\underline{\Z})_\star=\\[2ex]
\Phi^{\Z/p}(H\underline{\Z})\wedge \Phi^{\Z/p}(H\underline{\Z})_*[b,b^{-1}]=\\[2ex]
A_*[\sigma^{-2}, \rho^{-2}][b,b^{-1}].
\end{array}}
Thus, to the eyes of $\mod p$ homology, to get from $H\underline{\Z}$ to $BP\R$ on geometrical
fixed points, we must kill the $\rho^{-2}$-powers in \rref{ephicon2}, while preserving the other
generators. 

To describe how this is done, let us recall briefly the $p=2$ case.
As remarked in \cite{hk}, for $p=2$, a BP-tower construction of $BP\R$ on the level
of geometric fixed points is modelled on the cobar complex of $\Z/2[\rho^{-2}]$, which is the same
thing as the cohomological Koszul complex of the divided polynomial power algebra $\Z/2[\rho^{-2}]^\vee$,
which is an exterior algebra. Thus, the generators are in total degrees $2^i-1$, $i=1,2,\dots$, which is correct,
since the dimension of $v_i$ is $(2^i-1)(1+\alpha)$, and the $\alpha$ may be disregarded on the level 
of geometrical fixed points. In particular, for $p=2$, on geometrical fixed points, the $BP\R$ tower is split (i.e.
the maps are equivalences of cancelling wedge summands), and will work for any choice of $Q^\prime_n$ which
sends $\rho^{-2^n}\mapsto 1$ (up to multiplying by a power of $a$). We also note that modulo possible
error terms in the derived wedge, this condition, together with its dimension, essentially determines $Q^\prime_n$,
since on the good wedge, inverting $a$ is injective. 

For $p>2$, the construction is precisely analogous (the only difference being that for $p=2$, error terms in the derived
wedge do not matter, since we have an a priori model of $BP\R$ coming from $M\R$ via spectral algebra, while for
$p>2$, no a priori geometric model is known, so we circumvent the difficulty by working in Borel cohomology).

On the level of geometric fixed points, however, again, the $BP\R$-tower is modelled on the cohomological
Koszul complex of the divided polynomial power algebra 
\beg{ephicon3}{\Z/p[\rho^{-2}]^\vee.
}
Now \rref{ephicon3} is a tensor product of truncated polynomial algebras of the form
\beg{ephicon3a}{\Z/p[x]/x^p}
where $x$ is the dual of 
\beg{ephicon4}{
\rho^{-2p^{n-1}}, \; n=1,2,\dots
}
This means that our primary operation corresponds to an exterior generator of cohomological degree $1$ and topological degree
$2p^{n-1}$. (The total degree $2p^{n-1}-1$ corresponds to the geometric fixed point version of $v_n$.)
Then there is a secondary generator in cohomological degree $2$ and topological
degree $2p^n$, which gives total degree $2p^n-2$, corresponding
to the geometric fixed point version of $\Phi(v_n)$.

Recalling the short exact sequence of $\underline{\Z}$-modules
$$0\r \underline{\Z}\r \underline{\mathcal{L}}_p\r \widetilde{\mathcal{L}}_{p-1}\r 0,$$
when taking geometric fixed points (i.e. inverting $b$), $H\widetilde{L}_{p-1}$ becomes identified
with $\Sigma H\underline{\Z}$, so we see that ignoring the $\beta$-part, our primary operation
is at least of the right degree $2p^{n-1}$. 

To assert that we actually have the right element, we need to show how $Q^\prime_n$ cancels the element
\rref{ephicon3}. To this end, we recall the computation in \cite{st1} in Borel cohomology:
\begin{equation*}
\begin{array}{ll}
  \underline{\xi}_1  & =(\sigma^{-2}- \rho^{-2}) b^{1-p} \\
  \underline{\xi}_{n+1}  & = (\sigma^{2p^{n-1}-2p^n} \underline{\xi}_n -
    \rho^{-2}\xi_{n})b^{p^n-p^{n+1}}. \\
\end{array}
\end{equation*}
The first equation eliminates the term $\rho^{-2}$ modulo the allowed error terms in \rref{ephicon1},
while the first term of the second equation can be similarly used by induction to eliminate $\rho^{-2p^{n-1}}$.

We will see below how the transpotence also corresponds to $\Phi(v_n)$. However, since it is defined
by canceling the homology element corresponding to 
the ``Kudo element" $x^{p-1}y$ where $x$ is as in \rref{ephicon3a} and $y$ is its Koszul dual, we have no 
statement on the level of the dual Steenrod algebra in that case. Similarly as in the $p=2$ case, again,
on the level of geometric fixed points, our $BP\R$-tower splits, i.e. is described as a wedge of equivalences of
wedge summands.

We also note that the Borel cohomology pieces we add to our tower match \rref{ebpre}. Therefore, by the calculation of
Section \ref{sborel}, adding Borel cohomology pieces, on the Tate level, we recover the $b$-non-torsion part of the coefficients,
which matches the above proposed geometric fixed points. This shows that both constructions correspond correctly.

\vspace{3mm}

Now to continue our construction, smashing the cofibration sequence
$$S^{\beta-2}\r S^0\r T$$
with $H\widetilde{L}_{p-1}$, we have a cofibration sequence
\beg{eresol3}{
H\widetilde{\mathcal{L}}_{p-1}\r H\widetilde{\mathcal{L}}_{p-1}\wedge T\r \Sigma^{\beta-1}H\widetilde{\mathcal{L}}_{p-1}. 
}
The $\underline{\xi}_n$-multiple of $Q_n^\prime$ kills on the $R$-graded coefficients of the
middle term of the $|\underline{\xi}_n|$-suspension of \rref{eresol3} 
everything except the $|\underline{\xi}_n|$-suspension of a copy of $H\underline{\mathcal{L}}_p$
coming from the first term of \rref{eresol3}.
Thus, we obtain a k-invariant
\beg{eresol4}{
\overline{Q}_n:H\widetilde{\mathcal{L}}_{p-1}\{v_n\}\r H\underline{\mathcal{L}}_p\{v_n^2\}.
}

On the other hand, the fiber $F$ of \rref{eresol3} contains a copy of a fiber of 
\beg{eresol5}{
H\underline{\Z}\cdot\underline{\mu}_n\r H\widetilde{\mathcal{L}}_{p-1}\wedge T\cdot\widehat{\xi}_n\underline{\xi}_n^{p-2}
}
which has a factor of $H\underline{\Z}$. We define a k-invariant
\beg{eresol6}{
Q_n^{\prime\prime}:F\r H\underline{\Z}\cdot\Phi(v_n)
}
isomorphically on that factor. 

In more detail, we have a cofibration sequence
\beg{edetail1}{
H\underline{\Z}\r HT\r \Sigma^{\beta-1}H\underline{\Z}
}
and thus also
\beg{edetail2}{
H\widetilde{\mathcal{L}}_{p-1}\r HT\wedge_{H\underline{\Z}}H\widetilde{\mathcal{L}}_{p-1}\r \Sigma^{\beta-1}H\widetilde{\mathcal{L}}_{p-1}.
}
On the other hand, we have a cofibration sequence
\beg{edetail3}{
H\widetilde{\mathcal{L}}_{p-1}\r H\underline{\mathcal{L}}_p\r \Sigma^{2-\beta}H\underline{\Z},
}
or
\beg{edetail4}{
H\underline{\Z}\r \Sigma^{\beta-1}H\widetilde{\mathcal{L}}_{p-1}\r \Sigma^{-1}H\underline{\mathcal{L}}_p.
}
Taking the derived pullback of \rref{edetail2} via the first map \rref{edetail4}, we obtain a cofibration
sequence
\beg{edetail5}{
H\widetilde{\mathcal{L}}_{p-1}\r HX\r H\underline{Z}.
}
However, \rref{edetail5} must split since there is no essential $H\underline{\Z}$-module map
$$H\underline{\Z}\r\Sigma H\widetilde{\mathcal{L}}_{p-1}$$
(by dimensional reasons).

This is the {\em even pattern} in the Borel cohomology $\Z/p$-equivariant
Brown-Peterson resolution. An illustration is given in Figure 9 of the Appendix.

\vspace{3mm}

The {\em odd pattern} is obtained essentially by smashing this over $H\underline{\Z}$ with
$H\widetilde{\mathcal{L}}_{p-1}$, but the negligible term appears at a slightly different place.

The first k-invariant comes from composing the ``connecting map"
$$\widetilde{\mathcal{L}}_{p-1}\r \underline{\mathcal{L}}_p$$ 
with the second map \rref{eresol3}, applied to the
$\widehat{\xi}_n$-summand of \rref{eresol*} smashed over $H\underline{\Z}$ with 
$H\widetilde{\mathcal{L}}_{p-1}$. This gives a negligible k-invariant
$$\overline{\overline{Q}}:H\widetilde{\mathcal{L}}_{p-1}\r H\underline{\mathcal{L}}_p\{v_n\}$$
and its fiber $F^\prime$ has a quotient of $H\underline{\Z}\cdot\underline{\xi}_n$. We obtain
a k-invariant
\beg{eresol7}{
Q_n^{\prime}:F^\prime\r H\underline{\Z}\cdot v_n
}
by applying an isomorphism on this quotient. The argument follows \rref{edetail1}-\rref{edetail5}.

On the other hand, we can apply the second map \rref{eresol1} to the $\underline{\xi}_n^{p-2}\cdot\widehat{\xi}_n$
copy in \rref{eresol*} to obtain the k-invariant
$$Q_n^{\prime\prime}:F^{\prime\prime}\r H\widetilde{\mathcal{L}}_{p-1}\cdot \Phi(v_n)$$
where $F^{\prime\prime}$ is the fiber of \rref{eresol7}. 

An illustration of the odd pattern is given in Figure 10 of the Appendix.

\vspace{5mm}
\section{Comparison with Johnson-Wilson spectra}\label{sjw}

In this section, we shall investigate Johnson-Wilson spectra analogs of $BP\R$. Denote by
$I_n$ the ideal $(p,v_1,\dots,v_{n-1})\subset BP_*$.

\begin{proposition}\label{tetale}
The spectrum $(E(n)\wedge\dots \wedge E(n))^\wedge_{I_n}$ is an $A_\infty$-ring spectrum and
the space of its $A_\infty$-selfmaps is homotopically discrete. (The same statement holds for 
other ``flavors" of $E(n)$, notably $E_n$.)
\end{proposition}

\begin{proof}
We follow the method of Robinson \cite{robinson} and Baker \cite{baker}. 
We first recall that the ideal generated by $I_n$ by inclusions 
$$BP_*\r (BP\wedge\dots\wedge BP)_*$$
via one of the factors does not depend on the factor. (See \cite{ravenel}, Theorem A2.2.6.)

Next, let
$$E(n,k)=\bigwedge_{k} E(n),$$
$$\Sigma(n,k)_*=(\bigwedge_{k}E(n))_*/I_n.$$
One writes $\Sigma(n)_*=\Sigma(n,1)_*$.
The method of Robinson \cite{robinson}, Baker \cite{baker} formula (3.10) reduces this task to proving
a vanishing result for a certain $\mathcal{E}xt$-group
\beg{hochschildvanish}{\mathcal{E}xt^{s>0}_{\Sigma(n,2k)}(\Sigma(n,k),\Sigma(n,k))=0}
where $s$ denotes the cohomological degree. 

To define the $\mathcal{E}xt$-groups concerned, we denote by $E$ the smash of $k$ copies of $E(n)$,
the ring structore on $E_*E$ is given by the composition
\beg{ecompos}{
(E\wedge E)_*\otimes (E\wedge E)_*\r (E\wedge E\wedge E\wedge E)_*\r (E\wedge E)_*
}
where the first map is give by sending the first two coordinate to the first and last coordinate in the target,
and the last two coordinates into the middle two coordinates (without changing order internally),
and the second map is given by multiplication in the first two and the last two coordinates. The tensor
product is over $\Z$.The left $E_*E$-module structure on $E_*$ is defined analogously, replacing the
last two coordinates $E\wedge E$ in the source and middle two coordinates in the target of \rref{ecompos} 
by $E$. It is customary to refer to this construction as ``Hochschild cohomology"
(\cite{robinson}), even though this is not correct in full generality.

To prove \rref{hochschildvanish}, we recall from \cite{robinson} that
\beg{ehochvanish1}{\displaystyle
\Sigma(n)_*=\left(\bigotimes_{i\geq 1}\right)_{K(n)_*}\left(K(n)_*\prod \frac{K(n)_*[t_i]}{(v_nt_i^{p^n-1}-v_n^{p^i})}
\right)
}
In our present setting, we have duplicate copies of the coordinates $t_i$, but the essential point is that the $E_*E/I_n=
\Sigma(n,2k)$
forms a direct limit of \'{e}tale extensions of $K(n)_*$, and thus the given $\mathcal{E}xt^{>0}=0$.

\end{proof}

\vspace{5mm}
We claim that for all $n$, there is a $\Z/p$-action on $(\bigwedge_{p-1} E(n))^\wedge_{I_n}$ compatible
with the $\Z/p$-action on the non-equivariant spectrum underlying the $BP\R$ we constructed above. To this end,
we need to make a couple of remarks. First of all, the conjecture of Hill, Hopkins and Ravenel concerned the 
existence of a $\Z/p$-equivariant spectrum, which we denote by $BP\R^{HHR}$, which would satisfy
\beg{bprhhr}{
BP\R^{HHR}_{\{e\}}=\bigwedge_{p-1}BP.
}
Our construction gives the pattern \rref{ebpre} above.
We see \cite{aaa} that this gives the right answer at the prime $p=3$, but is smaller than the coefficients of 
\rref{bprhhr} for $p\geq 5$. One notes however that   
$BP\R^{HHR}$ can be obtained from $BP\R$ by taking a wedge with even suspensions
of copies of the additive norm $\Z/p_+\wedge BP$.

Now we claim that our construction 
determines a $\Z/p$-action on $(\bigwedge_{p-1}BP)_*$. which sends the universal formal group law 
to strictly isomorphic formal group laws (recall that a strict isomorphism between isomorphic formal group laws on
a torsion-free ring is uniquely determined). In a graded
ring $R\supseteq BP_*$, maps $BP_*\r R$ which carry
the universal formal group law to an isomorphic formal group law are characterized by elements $\overline{t}_i\in R$ of 
the same degree as $t_i$, which, when substituted for $t_i$ into $\eta_R(v_n)$, give the image of $v_n$.

In fact, we do not need the whole construction of $BP\R$, its first k-invariant \rref{eresol2}
determines the information we need. Non-equivariantly,
this gives an operation
\beg{enoneqq}{
H\Z\r \bigvee_{p-1}\Sigma^{2p^n-1}H\Z.
}
We know that this operation is just the integral $Q_n$ landing in an $H\Z$ wedge summand of the right hand side.
It is, further, invariant under $\Z/p$-action. This identifies the wedge summand which is supported by $v_n$.
In homotopy groups, on $\mathcal{L}_{p-1}$, it is an element which reduces modulo $p$ to an element
of $L_{p-1}$ which is annihilated by $1-\gamma$ where $\gamma$ is the generator of $\Z/p$.

Since, however, we also know from \cite{bp} that $v_n$ is only determined modulo $I_n$, we see that the
action on $v_n$ given by the $\Z/p$-action on \rref{enoneqq} necessarily sends the universal FGL to 
isomorphic FGL's.

Along with \rref{tetale}, this then implies 

\begin{theorem}\label{ttetale}
There exists a strict action of $\Z/p$ (in the ``naive" sense) on 
$$(\bigwedge_{p-1}E(n))^\wedge_{I_n}$$
(and its $I_n$-complete variants, replacing, for example, $E(n)$ by $E_n$) 
by morphisms of $A_\infty$-ring spectra, which  on coefficients
coincides with the action on compositions of $(p-2)$ strict isomorphisms of FGL's given by the
first k-invariant \rref{enoneqq}.
\end{theorem}
\qed

It is worth noting that while these considerations produce
 many equivalent actions on chains of $(p-2)$ strict isomorphisms
of FGL's, for $p>2$, there does not appear to be one canonical construction such as the $-i_F(x)$-series
associated with $BP\R$ for $p=2$.

\vspace{3mm}

It makes sense to conjecture that there exists an $A_\infty$-structure on $BP\R^{HHR}$ and
an $A_\infty$-map from $BP\R^{HHR}$ into the $\Z/p$-equivariant Borel-complete spectrum defined by 
Theorem \ref{ttetale}. At the moment, however, several ingredients are missing toward proving this, 
the first of which is the proof of the existence of an $A_\infty$-ring structure (or even a non-rigid ring structure)
on our construction of $BP\R^{HHR}$.

\vspace{3mm}

It is important to note that there also exists a $\Z/p$-action on $E_{p-1}$ induced from the 
the $\Z/p$-subgroup of the Morava stabilizer group. (For simplicity, let us assume $p\geq 5$.)
The $\Z/p$-action on coefficients is discussed in Nave \cite{nave}, Symonds \cite{symonds}.
Essentially, from the representation-theoretical point of view,
the non-negligible part is an exterior algebra on $\mathcal{L}_{p-1}$, tensored with
a polynomial algebra on $Nv_n$. (This is in no contradiction with the homotopy-theoretical point of view,
where we are dealing with the completion of a Laurent series ring.)

The generating $\mathcal{L}_{p-1}$ contains all
of 
\beg{enave}{u,u_{n-1}u,\dots, u_1u} 
where $u_i$ are the Lubin-Tate generators (i.e. are related to 
$v_i$ by change of normalization) and $u^{-1}$ is a root of $v_n$. In fact, it is proved in \cite{nave}
that $u$ generates this $\mathcal{L}_{p-1}$-representation and the elements \rref{enave}, 
in the order listed, are related by applying $1-\gamma$.

It was conjectured in \cite{hhr} that there should be a $\Z/p$ equivariant map from 
\beg{ehhrconj}{BP\R^{HHR}\r(E_{p-1})_{\Z/p}}
which would restict to an isomorphism of the $\mathcal{L}_{p-1}$-representations 
between the $\mathcal{L}_{p-1}$ containing $v_{1}$ and the $u^{-p^{p-1}+1}$-multiple
of \rref{enave}. Such a morphism of representations certainly exists, but making this assignment 
implies that the non-equivariant elements $v_i$ in our construction are sent to $0$ for $i>1$.
This may seem absurd, since the element $u^{-p^{p-1}+1}$ is of type $v_{p-1}$, and
is supposed to be inverted. The answer, however, is of course that there is no real contradiction, 
since when dealing with a chain of $p-2$ strict isomorphisms of the universal formal group law, there
could be many elements of type $v_{p-1}$, incluing some which contain a summand which is a unit
multiple of $t_{p-1}$. Thus, the element $u^{-p^{p-1}+1}$ could be an image of a generator of
a separate copy of $BP$ in $BP\R_{\{e\}}$. In fact, on the level of non-equivariant coefficients with
$\Z/p$-action, such maps are easily shown to exist. The existence of the comparison map \rref{ehhrconj}
on the level of $\Z/p$-equivariant spectra is at present still open.

\vspace{5mm}

\section{Odd prime real orientations and $E_\infty$-constructions}\label{seinf}

\vspace{3mm}

One knows (see \cite{lms}) that $MU$ is an $E_\infty$-ring spectrum. On the other hand, we have the
complex orientation
\beg{eeinf1}{\Sigma^{-2}\C P^\infty\r MU.}
Thus, (after appropriate discussion of cofibrancy), the universal property allows us to extend \rref{eeinf1} to
an $E_\infty$-map 
\beg{eeinf2}{C_\bullet\Sigma^{-2}\C P^\infty\r MU
}
where $C_\bullet E$ denotes the unital free $E_\infty$-ring spectrum on a spectrum $E$ with unit (meaning
a morphism $S\r E$). We remark in fact that \rref{eeinf2} is a retraction (i.e. that in the derived category of
spectra, the target splits off as a direct summand, as an ordinary ring spectrum). This is
due to the fact that $C_\bullet\Sigma^{-2}\C P^\infty$ is, by definition, complex-oriented, so the complex
orientation gives a section of \rref{eeinf2} in the category of ordinary commutative ring spectra.

\vspace{3mm}
A similar story is true, essentially without change, for Real-oriented $\Z/2$-equivariant spectra: Denoting by 
$\mathbb{S}^1$ the unit sphere in $\C$ with the $\Z/2$-action of complex conjugation, we have a Real
orientation
\beg{eeinf1a}{B\mathbb{S}^1\r M\R,}
which gives rise to a $\Z/2$-equivariant $E_\infty$-map
\beg{eeinf2a}{C_\bullet\Sigma^{-2}B\mathbb{S}^1\r M\R}
which, again, has a section in the category of ordinary $\Z/2$-equivariant commutative ring spectra (and hence,
in particular, splits in the category of $\Z/2$-equivariant spectra).

\vspace{3mm}
This raises the question what happens for a prime $p>2$. Hahn, Senger and Wilson \cite{oddorient}
defined a version of Real orientations of $\Z/p$-equivariant spectra for odd primes $p$ as follows. One denotes
$$\mathbb{T}=B_{\Z/p}\mathcal{L}_{p-1}.$$
The $\Z/p$-equivariant space $\mathbb{T}$ can, indeed, be identified with the subset of 
$(S^1)^{\Z/p}$ consisting of all tuples $(z_1,\dots,z_p)$ where 
$$\prod_{j=1}^p z_j=1.$$
Note that in particular, 
\beg{eeinf10}{\mathbb{T}^{\Z/p}=\Z/p.
}
Selecting once and for all any (non-equivariant path from $0$ to $1$ in \rref{eeinf10} specifies a 
$\Z/p$-equivariant map
\beg{eeinf11}{
\widetilde{\Z/p}\r \mathbb{T},
}
which, in turn, by adjunction, gives a $\Z/p$-equivariant map
\beg{eeinf12}{
\iota:\Sigma\widetilde{\Z/p}\r B\mathbb{T}.
}
The right-hand side is, again, equivalent whether whether we take $B_{\Z/p}$ or the ``naive" Schubert cell
construction. In fact, we have the following

\begin{proposition}\label{pfoling1}
Let $\mathbb{T} = \{(z_1, \cdots, z_p) | z_i \in S^1, \prod z_i = 1\}$ be the $\mathbb{Z}/p$-equivariant
group so that $\mathbb{Z}/p$ acts by permuting the coordinates. Then the equivariant
classifying space $B_{\mathbb{Z}/p} \mathbb{T}$ is equivalent to the fiber of 
\begin{equation}
  \label{eq:BT}
  B((S^1)^{\mathbb{Z}/p}) \to BS^1.
\end{equation}
In \eqref{eq:BT}, $B$ denotes the bar construction,
$(S^1)^{\mathbb{Z}/p}$ is the $\{(z_1, \cdots, z_p) |
z_i \in S^1\}$ and the map is induced by multiplying the coordinates.  
\end{proposition}

\vspace{3mm}
\noindent
{\bf Comment:} The bar construction is equivariant and also preserves the ``multiplicative norm" on space level,
so the source of \rref{eq:BT} is a product of $p$ copies of $\C P^\infty$ with action by permutation of coordinates,
while the target of \rref{eq:BT} is a fixed copy of $\C P^\infty$.

\begin{proof}
Notice that on fixed points, \eqref{eq:BT} is $p: BS^1 \to
BS^1$, (meaning the bar construction applied to the $p$'th power map on $S^1$), 
so it suffices to show that $(B_{\mathbb{Z}/p} \mathbb{T})^{\mathbb{Z}/p} \simeq B \mathbb{Z}/p$, as the underlying
level is easy. By \cite[Theorem 10]{LashofMay},
\begin{equation}
  \label{eq:LashofMay}
(B_{\mathbb{Z}/p} \mathbb{T})^{\mathbb{Z}/p} \simeq \coprod_{[\Lambda]} B W_{\mathbb{T} \rtimes \mathbb{Z}/p}(\Lambda),
\end{equation}
where the index is over subgroups $\Lambda \subset \mathbb{T} \rtimes \mathbb{Z}/p$ isomorphic to $\mathbb{Z}/p$ and
intersecting $\mathbb{T}$ trivially, up to $ \mathbb{T} \text{-conjugation}$.
Write $\mathbb{Z}/p = \langle \gamma \rangle$ and take $x = ((z_1, \cdots, z_p),\gamma) \in \mathbb{T} \rtimes \mathbb{Z}/p$. Then 
\begin{equation*}
x^p = (( \sqcap z_i, \cdots,  \sqcap z_i), \gamma^p) = e \text{ and } x^{-1} = ((z_2^{-1}, \cdots, z_{p}^{-1}, z_1^{-1}), \gamma^{-1}).
\end{equation*}
This means any element $x$ of $ \mathbb{T} \rtimes \mathbb{Z}/p$ not contained in $\mathbb{T}$ generates a $\Lambda$.
However, these are all conjugate. Let $y = ((y_1, \cdots, y_p), \gamma)$ to be
determined. Setting 
\begin{equation*}
yxy^{-1} = ((y_1x_py_p^{-1}, y_2x_1y_1^{-1}, \cdots, y_px_{p-1}y_{p-1}^{-1}), \gamma) =
((1,1,\cdots,1), \gamma),
\end{equation*}
we obtain
\begin{align*}
  & y_{p}/y_1 = x_p \\
  & y_1/y_2 = x_1 \\
  &  \cdots \\
  & y_{p-1}/y_p = x_{p-1}
\end{align*}
It is not hard to see that this allows a solution with $\prod y_i = 1$.
Now, using $\Lambda = \langle((1,1,\cdots,1), \gamma)\rangle$ without loss of generality,
one obtains $(B_{\mathbb{Z}/p}\mathbb{T})^{\mathbb{Z}/p} \simeq B \mathbb{Z}/p$ from \eqref{eq:LashofMay}.
\end{proof}

Note that in the proof, the only thing we used about $S^1$ is that it is an
abelian compact Lie group. However, the same argument often extends to more general groups.
For example, we can show that the fiber of 
$$B(\mathbb{Z}^{\mathbb{Z}/p}) \to B
\mathbb{Z}$$ 
is 
$B_{\mathbb{Z}/p} \mathcal{L}_{p-1},$
recalling  $\mathcal{L}_{p-1} = \{(a_1, \cdots, a_p) | z_i \in \mathbb{Z}, \sum
a_i = 0\}$. In other words, $\mathbb{T} \simeq B_{\mathbb{Z}/p} \mathcal{L}_{p-1}$.
Although the group $\Z$ is non-compact, the formalism works in this case, since the family consits of finite subgroups.
Instead of \cite[Theorem 10]{LashofMay}, we use the fact that $H^1(\Z/p,\mathcal{L}_{p-1})=\Z/p$, by
the long exact se	quence in cohomology induced by the short exact sequence
$$0\r\Z\r \mathcal{L}_{p}\r\mathcal{L}_{p-1}\r 0.$$

Hahn Senger, and Wilson \cite{oddorient} define a $\Z/p$-equivariant commutative ring spectrum $E$ to be
{\em Real-oriented} if the following diagram can be completed:
\beg{eeinf20}{
\diagram
\Sigma\widetilde{\Z/p}\dto_\iota\rto &\Sigma\widetilde{\Z/p}\wedge E\\
B\mathbb{T}\urdotted|>\tip &
\enddiagram
}
where the horizontal map is the smash of $\Sigma\widetilde{\Z/p}$ with the unit. If this happens, one can 
show that the Schubert-cell spectral sequence converging to $E^*B\mathbb{T}$ (see Proposition \ref{pfoling1})
collapses. 

In effect, we claim that for a Real-oriented $\Z/p$-equivariant spectrum $E$, modulo "negligible parts,"
i.e. copies of the additive norm from $\{e\}$ to $\Z/p$ on $E_{\{e\}}$, we have
\beg{esscollapse}{
E\wedge B\mathbb{T}\sim \bigwedge_E (E\vee E\wedge \Sigma \widetilde{\Z/p})\wedge_E E[Nx]
}
(we mean this up to homotopy, not in any coherent sense). In fact, the $E\wedge \Sigma \widetilde{\Z/p}$
occurs by the definition of $\Z/p$-equivariant Real orientation, while the class $Nx$ is simply the multiplicative norm
(on the level of spaces) of the complex orientation of $E_{\{e\}}$ (whose existence also follows from
the definition of Real orientation by forgetting the $\Z/p$-action), restricted by the inclusion map
$$B\mathbb{T}\r B((S^1)^{\Z/p}.$$

By the observation, the ring $E^*[Nx]$ splis off as a canonical summand of $E^*B\T$. Multiplication on $B\T$
then gives
rise to a $p$-valued formal group law in the sense of Buchstaber \cite{buch, buch1}, see Definition 1.2 of 
\cite{buch1}. To recapitulate, Buchstaber defines a $p$-valued formal group law as a polynomial of the
form
\beg{buchh1}{
\Theta(x,y)=Z^p-\theta_1(x,y)Z^{p-1}+\dots-\theta_p(x,y)
}
where $\theta_i(x,y)$ are power series which satisfies the conditions
$$\displaystyle \Theta(x,0)=Z^n+\sum_{i=1}^n(-1)^i{n\choose i}x^iZ^{n-i},$$

\beg{multiassoc}{\Theta(\Theta(x,y),z)=\Theta(x,\Theta(y,z)),}

$$\Theta(x,y)=\Theta(y,x).$$
The associativity condition \rref{multiassoc} requires some explanation.

If $F(x,y)$ is an ordinary formal group law, then a $p$-valued formal group law can be obtained as
$$\Theta(x,y)=\prod_{j=1}^n(Z-F_j(x,y))$$
where 
$$F_j(x,y)=
\exp_F((\log_Fx)^{1/p}+\zeta_p^j(\log_Fy)^{1/p})^p.$$
In the associativity condition, we can similarly write terms in three variables
\beg{multiassoc1}{F_{jk}(x,y,z)=
\exp_F((\log_Fx)^{1/p}+\zeta_p^j(\log_Fy)^{1/p}+\zeta_p^k(\log_Fz)^{1/p})^p
}
and the multivalued formal sum in three variables should be
\beg{etriple}{\Theta(x,y,z)=\prod_{j,k=1}^n(Z-F_{jk}(x,y,z)).}
This can be processed in two different ways into expressions involving only the series $\theta_j$:

We can write \rref{multiassoc1} as
$$\begin{array}{l}F_{jk}(x,y,z)=
\\
\exp_F((\log_Fx)^{1/p}+\zeta_p^j(((\log_Fy)^{1/p}+
\zeta_p^{k-j}(\log_Fz)^{1/p})^p)^{1/p})^p\end{array}$$
and thus express \rref{etriple} in terms of 
$$\theta_i(x,F_j(y,z)).$$
The resulting expression, however, is symmetrical in the $F_j(y,z)$, which allows them to be reduced to 
$\theta_s(y,z)$. Permuting the variables $x,y,z$ cyclically and doing the same thing gives another 
expression. The equality of both expressions is what one means by \rref{multiassoc}.

\vspace{3mm}

For $p=2$, Real-oriented spectra produce an ordinary formal group law due to the fact that the orientation class
is a polynomial generator. As we saw in \rref{esscollapse} (and as is, in some sense, the
theme of this paper), however, for $p>2$ the orientation class
is only an ``exterior generator" from the point of view of representation theory, which is why we have to
restrict to the polynomial ring on the (space-level) multiplicative norm of the non-equivariant orientation
class, which leads only to a $p$-valued formal group law. Note that the collapse \rref{esscollapse} is necessary
to assure that the ring $E^*[Nx]$ indeed survives as a canonical summand of the cell spectral
sequence given by Proposition \ref{pfoling1}.

\vspace{5mm}
For our present purposes, we recall from \cite{lms} that the 
$\Z/p$-equivariant Spanier-Whitehead dual of $\widetilde{\Z/p}$ is given by
\beg{eeinf30}{
D\widetilde{\Z/p}=\Sigma^{-\beta}\widetilde{\Z/p}.
}
Thus, 
$$\Sigma^{-1-\beta}\widetilde{\Z/p}\wedge B\mathbb{T}$$
becomes a unital spectrum via $\iota$ and the spectrum
\beg{eeinf31}{
C_\bullet \Sigma^{-1-\beta}\widetilde{\Z/p}\wedge B\mathbb{T}
}
becomes Real-oriented in the sense of \cite{oddorient}. It seems therefore reasonable to conjecture that
after localizing at $p$, $BP\R$ will split off \rref{eeinf31}. 

One also notes that since $(E_{p-1})_{\Z/p}$ is both $E_\infty$ and Real-oriented (as proved in
\cite{oddorient}), we do know that there is a $\Z/p$-equivariant $E_\infty$-map
$$C_\bullet \Sigma^{-1-\beta}\widetilde{\Z/p}\wedge B\mathbb{T}\r E_{p-1}.$$

\vspace{5mm}

\section{Appendix: Graphical illustration of the construction of $BP\R$ at $p>2$}

The purpose of this Appendix is to provide a graphical illustration of the construction given in Section \ref{sconst}. The even (resp. odd)
pattern of the tower is depicted in Figure 9 (resp. Figure 10). Let us briefly recall the classical Brown-Peterson $BP$-tower \cite{bp}. Their tower is
described by taking the homology $H\Z/p_*H\Z$ and killing those elements which are not supposed to be in $H\Z/p_*BP$. These are, of course,
all the copies of $H\Z/p_*=\Z/p$ indexed by a generator $\xi_R\tau_E$ in Milnor's notation \cite{milnor} where $E=(e_1,e_2,\dots)\neq 0$, $e_n\in\{0,1\}$. In the first stage of the tower, we attach copies of $H\Z$ along the Milnor primitives $Q_n$, (the corresponding generators being thus labelled
by $v_n$), which sends $\tau_n\mapsto 1$, thereby killing
all these elements, creating, however, ``error terms" due to the relations $\tau_n^2=0$. These are corrected in subsequent stages of
the tower, thus making the generators $v_n$ polynomial.

This construction has an
exact analogue in constructing $BP\R$ for $p=2$ \cite{hk}; while the relation for $\tau_n^2$ is more complicated, it does not affect the
multiplication rule of the $v_n$'s (this also occurs in the classical $BP$ tower for $p=2$).

In the case of the $BP\R$ tower for $p>2$ odd, we proceed similarly, with several complications. One is that unlike the non-equivariant case, 
where all the constituent pieces of the tower are suspensions of $H\Z$, or the $BP\R$ at $p=2$, where the pieces are $RO(\Z/2)$-graded suspensions
of $H\underline{\Z}$, in constructing $BP\R$ at $p>2$, we encounter $RO(\Z/p)$-graded suspensions of several different kinds of pieces,
namely $H\underline{\Z}$, $H\widetilde{\mathcal{L}}_{p-1}$,$H\underline{\Z/p}$, and $H\underline{\mathcal{L}}_p$. Thus, the general plan of our pictures is
to depict the homology of the generators involved, using the symbols introduced in Figures 2,4,6, 
while showing the k-invariants killing appropriate parts of their $RO(\Z/p)$-graded
coefficients (depicted by the curved arrows). We remind the reader that while generators of the form
$H\underline{\mathcal{L}}_p$ are necessary to complete the tower, such a generator represents a free spectrum, thus spawning 
a non-equivariant $BP$ tower, which is considered negligible and is omitted from the picture. This corresponds to higher powers of the $v_n$
generators being negligible, thus making the picture necessarily incomplete. Finally, our pictures depict the good wedges of the genuine 
Eilenberg-Mac Lane generators. We see that the primary k-invariants wipe out the ``odd part'' of the good wedges of generators
of the equivariant dual Steenrod algebra,
as well as the new homology of the generators arising in the tower, completely. The reader should recall, however, that due to the presence of
the derived wedges, this ``slice version'' of the tower at present remains conjectural, and we only have a Borel cohomology version. The reason
we keep the genuine equivariant homology picture is to demonstrate in what exact degrees the generator classes reside.

\vspace{5mm}

\subsection{Detailed description of the even pattern}

In Figure 9, the illustration is for the prime $p=3$. The generator in the bottom left corner is the bottom element $H\underline{\Z}$. The
triangle it is enclosed in depicts the fact that this generator survives. The leftmost column of Figure 9 depicts what the first k-invariant does
on the factor denoted by $T_\theta(t_i)$ in Theorem A of \cite{sankarw} (corresponding to the generator monomials in (4) of \cite{st1}).
Going from the bottom, the second icon shows the element $\widehat{\xi}_n$ which survives (and hence is encircled in a triangle). The rest of
this copy of
$H\Z\wedge T$ (in the notation of \cite{st1}) supports the k-invariant $Q_n^\prime$ going to the $H\widetilde{\mathcal{L}}_{p-1}$ corresponding
to the generator $v_n$. (We can imagine this part of the tower occurring multiplied by a polynomial in $\Phi(v_m)$ for any $m\in\N$
and an even monomial in $v_m$, $m>n$, occurring in the first power.) Now for the higher terms of the first column, a different pattern
occurs, since they are not supposed to contain any surviving coefficient elements. Accordingly, they are taken isomorphically
by the k-invariant $Q_n^\prime$  to the corresponding copy of $HX$ (the second icon, from the bottom, in the second column from the left 
of Figure 9). We observe that this leaves a copy of $H\underline{\mathcal{L}}_p$ in the target unattended, which supports a $\overline{Q}_n$
into a negligible part of the tower (not depicted in Figure 9). 

Since we are at $p=3$, the fourth icon (from the bottom) in the first column of Figure 9 is already the top term (see (4) of \cite{st1}),
supporting an $H\underline{\Z/p}$ according to Theorem A of \cite{sankarw}. This is enclosed in a dashed triangle, depicting two new phenomena, 
since only one generator in homology, denoted in \cite{st1} by $\underline{\theta}_n$, survives. This element, however, is integral, and accordingly,
is taken by $Q_n^\prime$ by a Bockstein to its target, indicating that no additional negligible term is spawned in this dimension (since the 
negligible higher powers of $v_n$ are already visible in the undepicted negligible part of our tower, starting with $v_n^2$). Also, we note that
the remainder of third icon (from the bottom) of the second column from the left of Figure 9 is also matched by $Q_n^\prime$ with half of the
coefficients of the top $H\underline{\Z/p}$ term of the first column, thus leaving an integral surviving pattern.

Finally, Figure 9 depicts a $Q_n^{\prime\prime}$ matching the top term of the second column counted from the left with the last column 
(which has only one icon). This is the transpotence k-invariant, supporting the $\Phi(v_n)$ generator. Note that, again, a copy of 
$H\underline{\mathcal{L}}_p$ is unattended, signifying another negligible generator in the tower.

This part of the tower is repeated by multiplication by powers of $\underline{\theta}_n$.

\begin{figure}
\includegraphics{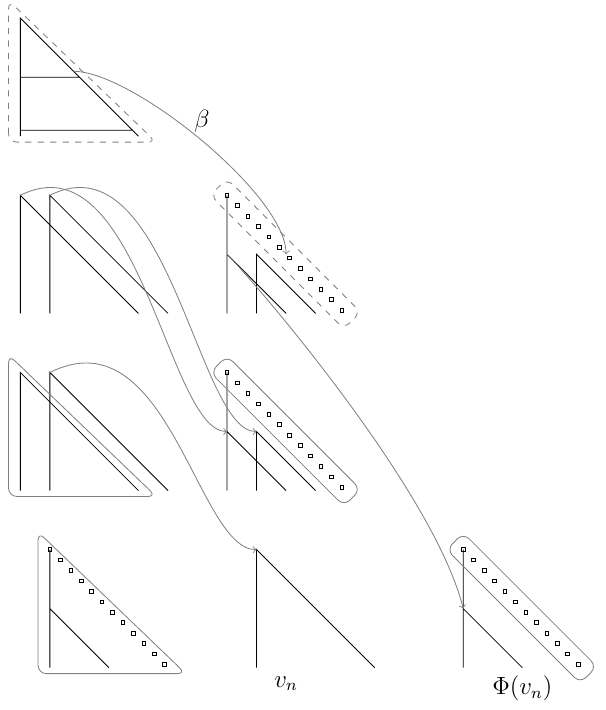}
\caption{The even pattern of the $BP\R$-tower}
\end{figure}

\begin{figure}
\includegraphics{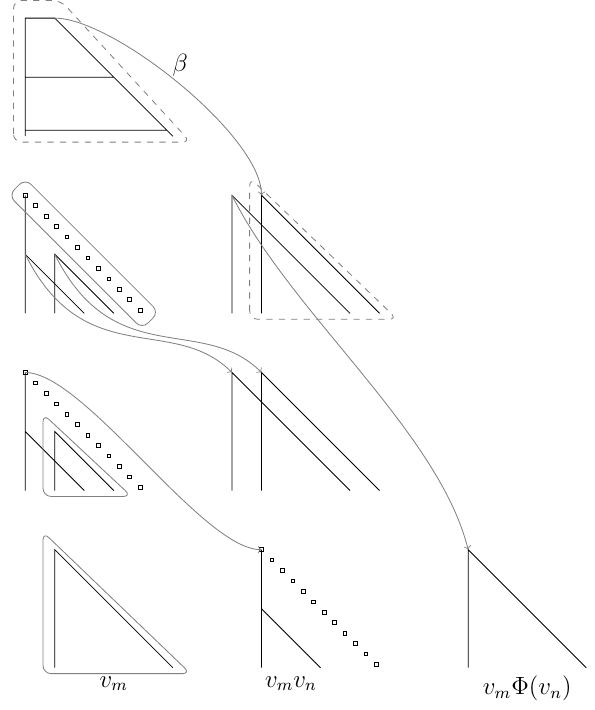}
\caption{The odd pattern of the $BP\R$-tower}
\end{figure}

\vspace{5mm}

\subsection{Detailed description of the odd pattern}

We now turn to Figure 10. We can think of this as continuation of Figure 9 with $n$ replaced by a larger number $m$, where the generator $v_m$, i.e. the surviving bottom
icon of the second column of Figure 9 representing an $H\widetilde{\mathcal{L}}_{p-1}$, is moved to the bottom left corner of Figure 10. 
The next two icons (from the bottom) of the first column of Figure 10 represent copies of $HX$ (see Section \ref{sconst}).
The differential $Q_n^\prime$ matches the coefficients of this term with the $H\underline{\Z}$ in the bottom of the second 
column of Figure 10, leaving a copy of $H\widetilde{\mathcal{L}}_{p-1}$ in the second icon of the first column, which survives in the 
homology of $BP\R$. The third icon from the bottom in the first column of Figure 10 is again an $HX$, this time matched by $Q_n^\prime$
with a target $H\underline{Z}\wedge T$ represented by the second icon from the bottom of the second column from the left of
Figure 10. This time, we see that there is a seemingly surviving copy of $H\underline{\mathcal{L}}_p$ in the source, which, however, 
supports another negligible part of the $BP\R$ tower, via the k-invariant $\overline{\overline{Q}}$.

At the top of the first column (counted from the left) of Figure 10, we have a copy of $H\widetilde{\mathcal{L}}_{p-1}/p$. 
This partially survives into a copy of $H\widetilde{\mathcal{L}}_{p-1}$ on the $\underline{\theta}_n$ generator. To achieve this effect,
there is a Bockstein into one $H\widetilde{\mathcal{L}}_{p-1}$ part of the target $H\underline{\Z}\wedge T$ in the third icon from the bottom
of the second column of Figure 10, realized by the k-invariant $Q_n^\prime$. 

The remaining copy of $H\widetilde{\mathcal{L}}_{p-1}$ is matched, again, by the ``transpotence'' differential $Q_{n}^{\prime\prime}$ with
the icon in the third column of Figure 10, representing $v_m\Phi(v_n)$.

\end{document}